\documentclass[a4paper,10pt]{amsart}

\usepackage{graphicx}
\usepackage{geometry}
\usepackage{amsthm}
\usepackage{amssymb}
\usepackage{amsmath}
\usepackage{hyperref}
\usepackage{xcolor}
\usepackage{enumerate}
\usepackage{listings}
\usepackage{complexity}
\usepackage{blkarray}
\usepackage{lmodern}

\usepackage[font=small,labelfont=bf]{caption}
\usepackage[font=small,labelfont=normalfont,labelformat=simple]{subcaption}
\graphicspath{{figs/}}

\newtheorem{theorem}{Theorem}

\newtheorem{lemma}[theorem]{Lemma}

\newtheorem{question}{Question}

\newtheorem{conjecture}{Conjecture}

\author
{
Raphael Steiner 
}
\thanks{Departement of Computer Science, Institute of Theoretical Computer Science, ETH Z\"{u}rich, Switzerland,  \texttt{raphaelmario.steiner@inf.ethz.ch}.
This work was supported by an ETH Postdoctoral Fellowship.}

\date{\today}

\title{Disproof of a Conjecture by Woodall}

\begin{document}
\maketitle

\begin{abstract}
In 2001, in a survey article~\cite{woodall} about list coloring, Woodall conjectured that for every pair of integers $s,t \ge 1$, all graphs without a $K_{s,t}$-minor are $(s+t-1)$-choosable.
In this note we refute this conjecture in a strong form: We prove that for every choice of constants $\varepsilon>0$ and $C \ge 1$ there exists $N=N(\varepsilon,C) \in \mathbb{N}$ such that for all integers $s,t $ with $N \le s \le t \le Cs$ there exists a graph without a $K_{s,t}$-minor and list chromatic number greater than $(1-\varepsilon)(2s+t)$. 
\end{abstract}

\section{Introduction}
\paragraph{\textbf{Preliminaries.}} All graphs considered in this paper are loopless and have no parallel edges. Given numbers $s,t \in \mathbb{N}$ we denote by $K_t$ the complete graph of order $t$ and by $K_{s,t}$ the complete bipartite graph with bipartition classes of size $s$ and $t$, respectively. Given graphs $G$ and $F$, we say that $G$ \emph{contains $F$ as a minor}, in symbols, $G \succeq F$, if there exists a collection of pairwise disjoint non-empty subsets $(Z_f)_{f \in V(F)}$ of the vertex-set of $G$ such that for every $f \in V(F)$, the induced subgraph $G[Z_f]$ of $G$ is connected, and furthermore, for every edge $f_1f_2 \in E(F)$, there exists at least one edge in $G$ with endpoints in $Z_{f_1}$ and $Z_{f_2}$. It is easily seen that our definition above is equivalent to the standard definition of graph minors, i.e., $G \succeq F$ if and only if $G$ can be transformed into a graph isomorphic to $F$ by performing a sequence of vertex or edge deletions, and edge contractions.

A \emph{proper coloring} of a graph $G$ with color-set $S$ is a mapping $c:V(G) \rightarrow S$ such that $c^{-1}(s)$ is an independent set, for every $s \in S$. A \emph{list assignment} for $G$ is an assignment $L:V(G) \rightarrow 2^\mathbb{N}$ of finite sets $L(v)$ (called lists) to the vertices $v \in V(G)$. An \emph{$L$-coloring} of $G$ is a proper coloring $c:V(G) \rightarrow \mathbb{N}$ of $G$ in which every vertex must be assigned a color from its respective list, i.e., $c(v) \in L(v)$ for every $v \in V(G)$. 
The chromatic number $\chi(G)$ of a graph $G$ is defined as the smallest integer $k \ge 1$ such that $G$ admits a proper coloring with color-set $[k]$. Similarly, the \emph{list chromatic number} $\chi_\ell(G)$ of a graph $G$ is defined as the smallest number $k \ge 1$ such that $G$ admits an $L$-coloring for \emph{every} assignment $L(\cdot)$ of color lists to the vertices of $G$, provided that $|L(v)| \ge k$ for every $v \in V(G)$ (we refer to this property by saying that $G$ is \emph{$k$-choosable}).

Clearly, $\chi(G) \le \chi_\ell(G)$ for every graph $G$, but in general $\chi_\ell(G)$ is not bounded from above by a function in $\chi(G)$, as shown by complete bipartite graphs. 

\bigskip

Hadwiger's conjecture, a vast generalization of the four-color-theorem~\cite{appelhaken1,appelhaken2} and arguably one of the most important open problems in graph theory, states the following upper bound on the chromatic number of graphs containing no $K_t$-minor:
\begin{conjecture}[Hadwiger~\cite{hadwiger}, 1943]\label{hadwiger}
For every $t \in \mathbb{N}$, if $G$ is a graph such that $G \not\succeq K_t$, then $\chi(G) \le t-1$.
\end{conjecture}
Hadwiger's conjecture has given rise to many beautiful results and open problems in the past. For a good overview of this field of research, encompassing the major results until about $2$ years ago, we refer the reader to the survey article~\cite{survey} by Seymour. Very recently, there has been considerable progress on the asymptotic version of Hadwiger's conjecture, see~\cite{norine, norine2, postle, postle2, postle3, del} for further reference. 

The remarkable difficulty of Hadwiger's conjecture has led to the study of several relaxations. One natural such relaxation is to prove the conjecture for graphs which (more strongly) exclude not only $K_t$, but a fixed sparser graph $H$ on $t$ vertices as a minor. In particular the case when $H$ is a complete bipartite graph has received attention. In his survey article~\cite{woodall} on list coloring, Woodall made the following two conjectures, both of which have remained open problems thus far. The first conjecture is a weakening of Hadwiger's conjecture, that was independently proposed by Seymour (cf.~\cite{kostochkabip3}). 
\begin{conjecture}[cf.~\cite{woodall}, Conjecture ``$C(r,s,\chi)$'']\label{con:chrom}
For every $s, t \in \mathbb{N}$, if $G$ is a graph with $G \not\succeq K_{s,t}$, then $\chi(G) \le s+t-1$.
\end{conjecture}

\begin{conjecture}[cf.~\cite{woodall}, Conjecture ``$C(r,s,\text{ch})$'']\label{con:lists}
For every $s, t \in \mathbb{N}$, if $G$ is a graph with $G \not\succeq K_{s,t}$, then $\chi_\ell(G) \le s+t-1$.
\end{conjecture}

For both conjectures, several partial results have been obtained in the past, which we summarize in the following.

Woodall proved in~\cite{woodall,woodall2} that Conjecture~\ref{con:lists} holds if $s \le 2$. The correctness for $s=2$ can also be obtained as a consequence of the main result of Chudnovsky et al.~\cite{chudnovsky}: They proved, extending a result by Myers~\cite{myers}, that every $n$-vertex-graph with no $K_{2,t}$-minor has at most $\frac{t+1}{2}(n-1)$ edges. Hence the average degree of such graphs is less than $t+1$. The latter implies that every graph with no $K_{2,t}$-minor is $t$-degenerate and therefrom (via greedy coloring) $(t+1)$-choosable.

For $s=t=3$, Conjecture~\ref{con:lists} was confirmed by Woodall in~\cite{woodall}. The conjecture is also valid for $s=3$ and $t=4$, since J{\o}rgensen proved (cf. Corollary~7 in~\cite{jorgensen}) that every $K_{3,4}$-minor free graph is $5$-degenerate, and hence $6$-choosable. For $s=3$ and large values of $t$, it was proved in~\cite{kostochkabip3} by Kostochka and Prince that for $t \ge 6300$ every $n$-vertex graph without a $K_{3,t}$-minor has at most $\frac{t+3}{2}(n-2)+1$ edges. Therefore (by considering the average degree), every such graph is $(t+2)$-degenerate. Using greedy coloring one can conclude that every graph without a $K_{3,t}$-minor is $(3+t)$-choosable for $t \ge 6300$, which misses Woodall's conjecture by an additive constant of $1$ only.  With an additional argument\footnote{which does not seem to extend to list coloring}, Kostochka and Prince proved in~\cite{kostochkabip3} that Conjecture~\ref{con:chrom} holds for $s=3$ and $t \ge 6300$. 

Finally, the case $s=4 \le t$ was considered by Kawarabayashi~\cite{kawarabayashi}, who proved that graphs without a $K_{4,t}$-minor are $4t$-choosable, for every $t \ge 1$. For $s=t=4$, a better result is known: J{\o}rgensen proved (cf. Corollary~6 in~\cite{jorgensen}) that every graph without a $K_{4,4}$-minor is $7$-degenerate, and hence $8$-choosable.  

Let us now turn to asymptotic bounds for large $s$ and $t$. By recent results of Delcourt and Postle~\cite{del}, every graph with no $K_t$-minor is $O(t \log \log t)$-colorable and $O(t (\log \log t)^2)$-choosable. This in particular means that the maximum (list) chromatic number of graphs with no $K_{s,t}$-minor is bounded by $O((s+t)\log \log (s+t))$ (respectively $O((s+t)(\log \log (s+t))^2)$). These are the asymptotically best known bounds for Conjecture~\ref{con:chrom} and~\ref{con:lists} when $s$ and $t$ are of comparable size, however, if $t$ is significantly bigger than $s$, then better bounds are known. Notably, Kostochka~\cite{kostochkabip1,kostochkabip2} proved that for every $s \ge 1$, Conjecture~\ref{con:chrom} holds if $t \ge t_0(s)$, where $t_0(s)=O(s^3\log^3(s))$. However, an analogous result is not known for the list chromatic number. The only similar result in this direction was proved in~\cite{kostochkadens1,kostochkadens2}, see also~\cite{kuehn} for related results:  If $t$ is huge compared to $s$, concretely if if $t>(240s\log_2(s))^{8s\log_2(s)+1}$, then every graph with no $K_{s,t}$-minor is $(3s+t)$-degenerate, and therefore $(3s+t+1)$-choosable. 

\medskip 

In this note, we disprove Conjecture~\ref{con:lists}, by constructing counterexamples for sufficiently large values of $s$ and $t$ which have a comparable size. 
\begin{theorem}\label{main}
For every choice of constants $\varepsilon>0$ and $C \ge 1$ there exists $N=N(\varepsilon,C) \in \mathbb{N}$ such that for all integers $s,t $ with $N \le s \le t \le Cs$ there exists a graph without a $K_{s,t}$-minor and list chromatic number greater than $(1-\varepsilon)(2s+t)$.
\end{theorem}
For instance, if $s=t$, then the above implies that the maximum list chromatic number of graphs with no $K_{t,t}$-minor is at least as $3t-o(t)$, which substantially exceeds the conjectured upper bound of $2t-1$ for large $t$. It remains an interesting open problem whether Conjecture~\ref{con:lists} remains true if $s$ is fixed and $t$ is sufficiently large in terms of $s$. It would also be interesting to determine the smallest values of $s$ and $t$ (or of $s+t$) for which Conjecture~\ref{con:lists} fails, since the bounds coming from our proof of Theorem~\ref{main} are almost astronomical. In that regard, it could be interesting to study the following smallest open cases of Conjecture~\ref{con:lists}.

\begin{question}
Is every graph without a $K_{4,4}$-minor $7$-choosable? Is every graph without a $K_{3,5}$-minor $7$-choosable?
\end{question}

Regarding the true asymptotics of the list chromatic number of graphs with no $K_{s,t}$-minor, the following natural problem arises.

\begin{question}
Is it true that for all integers $1 \le s \le t$, every graph $G$ with $G \not\succeq K_{s,t}$ satisfies $\chi_\ell(G) \le 2s+t$?
\end{question}

In the remainder of this note we present the proof of Theorem~\ref{main}. It is probabilistic and relies on a few modifications of an argument previously used by the author in~\cite{steiner} to prove that the maximum list chromatic number of graphs without a $K_t$-minor is at least $2t-o(t)$, addressing a conjecture by Kawarabayashi and Mohar~\cite{kawarabayashimohar}, see also~\cite{op}, known as the \emph{List Hadwiger Conjecture}. 

\section{Proof of Theorem~\ref{main}}

Following standard notation, for a pair of natural numbers $m, n \in \mathbb{N}$ and a probability $p \in [0,1]$, we denote by $G(m,n,p)$ the bipartite Erd\H{o}s-Renyi graph, that is, a random bipartite graph $G$ with bipartition $A ; B$ such that $|A|=m$, $|B|=n$, and in which every pair $ab$ with $a \in A, b \in B$ is selected as an edge of $G$ with probability $p$, independently from all other such pairs. 

\begin{lemma}\label{random}
Let $\varepsilon \in (0,1)$, $C \ge 1$, $f \in \mathbb{N}$ and $\delta \in (0,1)$ be constants such that $f^2 \delta<1$. For every $n \in \mathbb{N}$, let $p=p(n):=n^{-\delta}$ and $m=m(n):=\lfloor Cn \rfloor$. Then with probability tending to $1$ as $n \rightarrow \infty$, the random graph $G=G(m(n),n,p(n))$ with bipartition $A \cup B$ simultaneously satisfies the following two properties:
\begin{itemize}
\item For every collection of pairwise disjoint non-empty sets $X_1,\ldots,X_k \subseteq A$, $Y_1,\ldots,Y_k \subseteq B$ such that $k \ge \varepsilon n$ and $\max\{|X_1|,\ldots,|X_k|,|Y_1|,\ldots,|Y_k|\} \le f$, there exists a pair of indices $(i,j) \in [k] \times [k]$ such that $G$ contains all the edges $xy, (x,y) \in X_i \times Y_j$.
\item $G$ has maximum degree at most $\varepsilon n$. 
\end{itemize}
\end{lemma}
\begin{proof}
\noindent
It will clearly be sufficient to prove that for each of the two events above individually, the probability for them not to occur tends to $0$ as $n \rightarrow \infty$. It then follows using a union bound that also the probability that at least one of the two events does not occur tends to $0$ as $n \rightarrow \infty$, proving the claim of the lemma.
\begin{itemize}
\item Let us first consider the probability event $E_n$ that $G$ does not satisfy the property claimed by the first item. We want to show that $\mathbb{P}(E_n) \rightarrow 0$ as $n \rightarrow \infty$. So consider a fixed collection $X_1,\ldots,X_k \subseteq A, Y_1,\ldots,Y_k \subseteq B$ of disjoint non-empty sets, where $k \ge \varepsilon n$ and $\max\{|X_1|,\ldots,|X_k|,|Y_1|,\ldots,|Y_k|\} \le f$. Let $E(X_1,\ldots,X_k,Y_1,\ldots,Y_k)$ be the probability event ``there exists no pair $(i,j) \in [k] \times [k]$ such that all the edges $xy$ with $x \in X_i$ and $y \in Y_j$ are included in $G$''. Fixing a pair of indices $(i,j) \in [k] \times [k]$, clearly the probability of the event that ``$X_i$ is not fully connected to $Y_j$'' equals $1-p^{|X_i||Y_j|} \le 1-p^{f^2}$.  Since these events are independent for different choices of $(i,j)$, it follows that 
$$\mathbb{P}(E(X_1,\ldots,X_k,Y_1,\ldots,Y_k)) \le (1-p^{f^2})^{k^2}$$ $$\le (1-p^{f^2})^{\varepsilon^2n^2}\le \exp(-p^{f^2}\varepsilon^2 n^2)=\exp(-\varepsilon^2n^{2-f^2\delta}).$$
With a (very) rough estimate, there are at most $$(m+n+1)^{m+n}=\exp(\ln(m+n+1)(m+n)) \le \exp(\ln((C+1)n+1)(C+1)n)$$ different ways to select the sets $X_1,\ldots,X_k,Y_1,\ldots,Y_k$. Hence, applying a union bound we find that 
$$\mathbb{P}(E_n) \le \exp(\ln((C+1)n+1)(C+1)n-\varepsilon^2n^{2-f^2\delta}).$$
The right hand side of the above inequality tends to $0$ as $n \rightarrow \infty$, since $f^2\delta<1$ and hence $\varepsilon^2n^{2-f^2\delta}=\Omega(n^{2-f^2\delta})$ grows faster than $\ln((C+1)n+1)(C+1)n=O(n\ln n)$. This proves that $G$ satisfies the properties claimed by the first item w.h.p., as required.
\item To show that also the property claimed by the second item holds true w.h.p., consider the probability that a fixed vertex $x \in A \cup B$ has more than $\varepsilon n$ neighbors in $G$. The degree of $x$ in $G(m,n,p)$ is distributed like a binomial random variable $B(n,p)$ if $x \in A$ and like $B(m,p)$ if $x \in B$. Hence the expected degree of $x$ is $np=n^{1-\delta}$ if $x \in A$ and $mp \in [n^{1-\delta},Cn^{1-\delta}]$ if $x \in B$. Hence, $\mathbb{E}(d_G(x))$ is smaller than $\frac{\varepsilon n}{2}$ for $n$ sufficiently large in terms of $\varepsilon$, $\delta$ and $C$. Applying Chernoff's bound we find for every sufficiently large $n$:
$$\mathbb{P}(d_G(x)>\varepsilon n ) \le \mathbb{P}(d_G(x)>2\mathbb{E}(d_G(x))) \le \exp\left(-\frac{1}{3}\mathbb{E}(d_G(x))\right) \le \exp\left(-\frac{1}{3}n^{1-\delta}\right).$$
Since this bound holds for every choice of $x \in A \cup B$, applying a union bound we find that the probability that $G$ has a vertex of degree more than $\varepsilon n$ is at most
$$(m+n)\exp\left(-\frac{1}{3}n^{1-\delta}\right) \le \exp\left(\ln((C+1)n)-\frac{1}{3}n^{1-\delta}\right)$$ which tends to $0$ as $n \rightarrow \infty$, as desired (here we used that $\delta<1$ and hence $n^{1-\delta}$ grows faster than $\ln((C+1)n)$). 
%it suffices to note that $G(n,n,p)$ can be embedded as a subgraph into the Erd\H{o}s-Renyi random graph $G(2n,p)$. Since $p=p(n)=n^{-\delta} \rightarrow 0$ for $n \rightarrow \infty$, it follows from well-known results on random graphs (compare e.g.~\cite{bollobas}, Chapter VII, Section 4, page 135, Theorem~9) that the maximum degree of $G(2n,p)$ (and hence of $G(n,n,p)$) is less than $\varepsilon n$ w.h.p.
\end{itemize}

\end{proof}

In the next intermediate result we derive a useful deterministic statement from Lemma~\ref{random} about the existence of graphs with certain properties, which then come in handy when we construct the lower-bound examples for Theorem~\ref{main}.

\begin{lemma}\label{cor}
For every $\varepsilon \in (0,1)$ and $C \ge 1$, there exists $n_0=n_0(\varepsilon,C)$ such that for all integers $m,n$ satisfying $n_0 \le n \le m \le Cn$, there exists a graph $H$ whose vertex-set $V(H)=A \cup B$ is partitioned into two disjoint sets $A$ of size $m$ and $B$ of size $n$, and such that the following properties hold:
\begin{itemize}
\item Both $A$ and $B$ form cliques of $H$,
\item every vertex in $H$ has at most $\varepsilon n$ non-neighbors in $H$, and
\item for all integers $1 \le s \le t$ such that $n \le s$ and $m \le (1-2\varepsilon)(s+t)$, $H$ does not contain $K_{s,t}$ as a minor. 
\end{itemize}
\end{lemma}
\begin{proof}
Let $f:=\lceil \frac{C}{\varepsilon} \rceil \in \mathbb{N}$ and $\delta:=\frac{\varepsilon^2}{4C^2}$. Then $f^2\delta <1$, and hence we may apply Lemma~\ref{random}. It follows directly that there exists $n_0=n_0(\varepsilon,C) \in \mathbb{N}$ such that for every $n \ge n_0$ there exists a bipartite graph $G'$, whose bipartition classes $A'$ and $B'$ are of size $\lfloor Cn \rfloor$ and $n$ respectively, and such that the following hold: 
\begin{itemize}
\item For every collection of pairwise disjoint non-empty sets $X_1,\ldots,X_k \subseteq A', Y_1,\ldots,Y_k \subseteq B'$ such that $k \ge \varepsilon n$ and $\max\{|X_1|,\ldots,|X_k|,|Y_1|,\ldots,|Y_k|\} \le f$, there exists a pair $(i,j) \in [k] \times [k]$ such that $G'$ contains all the edges $xy, (x,y) \in X_i \times Y_j$.
\item $G'$ has maximum degree at most $\varepsilon n$. 
\end{itemize}
Since $n \le m \le \lfloor Cn \rfloor$, we may select and fix a subset $A \subseteq A'$ such that $|A|=m$. Also, put $B:=B'$. In the following, let $G:=G'[A \cup B]$ denote the induced subgraph of $G'$ with bipartition $A; B$.
We now define $H$ as the complement of $G$ (also with vertex-set $A \cup B$). It is clear from the definition of $G$ that $A$ and $B$ form cliques in $H$ and have the required size, verifying the first item in the claim of the lemma. The second item follows directly from the fact that $\Delta(G) \le \Delta(G') \le \varepsilon n$.

It hence remains to verify the last item. Towards a contradiction, suppose that there exist numbers $1 \le s \le t $ with $n \le s$, $m \le (1-2\varepsilon)(s+t)$, such that $H$ contains $K_{s,t}$ as a minor. This implies that there exists a collection $\mathcal{Z}=\mathcal{Z}_1 \cup \mathcal{Z}_2$ of non-empty and pairwise disjoint subsets of $V(H)$ such that $|\mathcal{Z}_1|=s$, $|\mathcal{Z}_2|=t$ and such that for every pair $Z_1 \in \mathcal{Z}_1, Z_2 \in \mathcal{Z}_2$, there exists at least one edge in $H$ connecting a vertex in $Z_1$ to a vertex in $Z_2$. 

Let us now consider $\mathcal{Z}_{A,1}:=\{Z \in \mathcal{Z}_1|Z \cap A \neq \emptyset\}$ and $\mathcal{Z}_{A,2}:=\{Z \in \mathcal{Z}_2|Z \cap A \neq \emptyset\}$. Since the sets in $\mathcal{Z}$ are pairwise disjoint, we can see that $|\mathcal{Z}_{A,1}|+|\mathcal{Z}_{A,2}| \le |A|=m \le (1-2\varepsilon)(s+t)$. We therefore must have $|\mathcal{Z}_{A,1}| \le (1-2\varepsilon)s$ or $|\mathcal{Z}_{A,2}| \le (1-2\varepsilon)t$. In the following, we distinguish these two cases and lead both to a contradiction. This will then show that our assumption on the existence of $s$ and $t$ above was incorrect, and hence complete the proof that $H$ satisfies all three properties required by the lemma. 

\medskip

\textbf{Case 1.} Suppose first that $|\mathcal{Z}_{A,1}| \le (1-2\varepsilon)s$. Then this means that $|\mathcal{Z}_1 \setminus \mathcal{Z}_{A,1}| \ge 2\varepsilon s \ge 2\varepsilon n$. The sets in $\mathcal{Z}_1 \setminus \mathcal{Z}_{A,1}$ are exactly those $Z \in \mathcal{Z}_1$ such that $Z \subseteq B$. Since $|B|=n$, and since the sets in $\mathcal{Z}_1 \setminus \mathcal{Z}_{A,1}$ are pairwise disjoint, it follows that $\mathcal{Z}_1 \setminus \mathcal{Z}_{A,1}$ contains at most $\varepsilon n$ sets of size more than $\frac{1}{\varepsilon}$. Consequently, at least $2\varepsilon n -\varepsilon n=\varepsilon n$ sets in $\mathcal{Z}_1 \setminus \mathcal{Z}_{A,1}$ have size at most $\frac{1}{\varepsilon} \le f$. Fix a list $Y_1,\ldots,Y_k \subseteq B$ of $k=\lceil \varepsilon n \rceil$ distinct sets in $\mathcal{Z}_1\setminus \mathcal{Z}_{A,1}$, each of size at most $f$. 

Next, consider the set $\mathcal{Z}_{B,2}:=\{Z \in \mathcal{Z}_2|Z \cap B \neq \emptyset\}$. Since the elements of $(\mathcal{Z}_1\setminus \mathcal{Z}_{A,1}) \cup \mathcal{Z}_{B,2}$ are pairwise disjoint and all intersect $B$, it follows that $|\mathcal{Z}_1\setminus \mathcal{Z}_{A,1}|+|\mathcal{Z}_{B,2}| \le |B|=n$, and hence that $|\mathcal{Z}_{B,2}| \le n-|\mathcal{Z}_1\setminus \mathcal{Z}_{A,1}| \le n-2\varepsilon n$. Since $t \ge s \ge n$, we conclude that $|\mathcal{Z}_2\setminus\mathcal{Z}_{B,2}| \ge t-(n-2\varepsilon n)=t-n+2\varepsilon n \ge 2 \varepsilon n$. All the sets $Z \in \mathcal{Z}_2\setminus \mathcal{Z}_{B,2}$ are fully included in $A$. Therefore, and since $|A|=m \le Cn$, there can be at most $\varepsilon n$ sets in $\mathcal{Z}_2\setminus \mathcal{Z}_{B,2}$ whose size exceeds $\frac{C}{\varepsilon}$. Hence, at least $2\varepsilon n-\varepsilon n=\varepsilon n$ sets in $\mathcal{Z}_2\setminus \mathcal{Z}_{B,2}$ have size at most $\frac{C}{\varepsilon} \le f$. Let $X_1,\ldots,X_k \subseteq A$ be $k=\lceil \varepsilon n\rceil$ distinct sets in $\mathcal{Z}_2\setminus \mathcal{Z}_{B,2}$, each of size at most $f$. By the property of $G'$ listed in the beginning of this proof, we know that there exists a pair $(i,j) \in [k] \times [k]$ such that all the edges $xy, (x,y) \in X_i \times Y_j$ are contained in $G'$ (and hence in $G$). This, however, means that there exists no edge in $H$ which connects a vertex in $X_i \in \mathcal{Z}_2$ to a vertex in $Y_j \in \mathcal{Z}_1$, contradicting our initial assumptions on the collection $\mathcal{Z}=\mathcal{Z}_1 \cup \mathcal{Z}_2$. This contradiction concludes the proof in Case~1.

\medskip 

The analysis for Case~2 below is almost identical to the analysis of Case~1, except for the fact that $\mathcal{Z}_1$ and $\mathcal{Z}_2$ play interchanged roles\footnote{We avoided reducing the second case to the first with a simple ``w.l.o.g.'' assumption, since formally the cases are not excatly symmetric (note that the sizes $s$ and $t$ of $\mathcal{Z}_1$ and $\mathcal{Z}_2$ can be different)}. 

\medskip

\textbf{Case 2.} Suppose now that $|\mathcal{Z}_{A,2}| \le (1-2\varepsilon)t$. This implies $|\mathcal{Z}_2 \setminus \mathcal{Z}_{A,2}| \ge 2\varepsilon t \ge 2\varepsilon s \ge 2\varepsilon n$. Note that all the sets in $\mathcal{Z}_2 \setminus \mathcal{Z}_{A,2}$ are fully included in $B$. Since the sets in $\mathcal{Z}_2 \setminus \mathcal{Z}_{A,2}$ are pairwise disjoint, it follows that there are at most $\varepsilon n$ sets of size more than $\frac{1}{\varepsilon}$ in $\mathcal{Z}_2 \setminus \mathcal{Z}_{A,2}$. Thus there are at least $2\varepsilon n -\varepsilon n=\varepsilon n$ sets in $\mathcal{Z}_2 \setminus \mathcal{Z}_{A,2}$ of size at most $\frac{1}{\varepsilon} \le f$. Fix $k=\lceil \varepsilon n \rceil$ distinct sets $Y_1,\ldots,Y_k \subseteq B$ in $\mathcal{Z}_2\setminus \mathcal{Z}_{A,2}$, each of size at most $f$. 

Let us now consider the subcollection $\mathcal{Z}_{B,1}:=\{Z \in \mathcal{Z}_1|Z \cap B \neq \emptyset\}$. The elements of $\mathcal{Z}_{B,1} \cup (\mathcal{Z}_2\setminus \mathcal{Z}_{A,2})$ all intersect $B$, and hence $|\mathcal{Z}_{B,1}|+|\mathcal{Z}_{2}\setminus \mathcal{Z}_{A,2}| \le |B|=n$. Therefrom, we have $|\mathcal{Z}_{B,1}| \le n-|\mathcal{Z}_2\setminus \mathcal{Z}_{A,2}| \le n-2\varepsilon n$. Since $s \ge n$, it follows that $|\mathcal{Z}_1\setminus\mathcal{Z}_{B,1}| \ge s-(n-2\varepsilon n)=s-n+2\varepsilon n \ge 2 \varepsilon n$. By definition, every set $Z \in \mathcal{Z}_1\setminus \mathcal{Z}_{B,1}$ is fully included in $A$. Due to $|A|=m \le Cn$, this implies that there are at most $\varepsilon n$ sets in $\mathcal{Z}_1\setminus \mathcal{Z}_{B,1}$ whose size exceeds $\frac{C}{\varepsilon}$. Hence, at least $2\varepsilon n-\varepsilon n=\varepsilon n$ sets in $\mathcal{Z}_1\setminus \mathcal{Z}_{B,1}$ have size at most $\frac{C}{\varepsilon} \le f$. Let $X_1,\ldots,X_k \subseteq A$ be $k=\lceil \varepsilon n\rceil$ distinct sets in $\mathcal{Z}_1\setminus \mathcal{Z}_{B,1}$, each of size at most $f$. Applying the first of the two properties of $G'$ as mentioned above, it follows that there must be $(i,j) \in [k] \times [k]$ such that all pairs $xy, (x,y) \in X_i \times Y_j$ are contained as edges in $G$. Since $H$ and $G$ are complements, this means that there exists no edge in $H$ which connects a vertex in $X_i \in \mathcal{Z}_1$ to a vertex in $Y_j \in \mathcal{Z}_2$, contradicting our initial assumptions. This contradiction concludes the proof also in Case~2.
\end{proof}

We are now almost ready for proving Theorem~\ref{main}. The only remaining ingredient is a simple lemma, which states that glueing two $K_{s,t}$-minor-free graphs together along a sufficiently small clique separator results in a graph that is again $K_{s,t}$-minor-free. We strongly suspect that this statement has appeared elsewhere before, but we decided to include the (simple) proof here for the reader's convenience. 

\begin{lemma}\label{glue}
Let $1 \le s \le t$ be integers, and let $G_1$ and $G_2$ be graphs not containing $K_{s,t}$ as a minor. Let $C:=V(G_1) \cap V(G_2)$. If $C$ forms a clique in both $G_1$ and $G_2$, and if $|C|<s$, then the graph $G_1 \cup G_2$ also does not contain $K_{s,t}$ as a minor. 
\end{lemma}
\begin{proof}
Towards a contradiction, suppose that $G:=G_1 \cup G_2$ contains $K_{s,t}$ as a minor. By definition, this means that there exists a collection of disjoint non-empty subsets $\mathcal{Z}=\mathcal{Z}_1 \cup \mathcal{Z}_2$ of $V(G_1) \cup V(G_2)$ such that $G[Z]$ is connected for every $Z \in \mathcal{Z}$, $|\mathcal{Z}_1|=s, |\mathcal{Z}_2|=t$, and for every pair $X \in \mathcal{Z}_1, Y \in \mathcal{Z}_2$, there exists an edge of $G$ with endpoints in $X$ and $Y$. Since $|C|<s=|\mathcal{Z}_1| \le |\mathcal{Z}_2|$, there exist $Z_1 \in \mathcal{Z}_1$ and $Z_2 \in \mathcal{Z}_2$ such that $Z_1 \cap C=Z_2 \cap C=\emptyset$. Since $Z_1$ and $Z_2$ induce connected subgraphs of $G$, this means that for both $i \in \{1,2\}$ we have $Z_i \subseteq V(G_1) \setminus C$ or $Z_i \subseteq V(G_2) \setminus C$ (since no edge in $G$ connects a vertex in $V(G_1) \setminus C$ to a vertex in $V(G_2) \setminus C$). Furthermore, by assumption there exists an edge with endpoints in $Z_1$ and $Z_2$, which implies that either $Z_1, Z_2 \subseteq V(G_1) \setminus C$, or $Z_1, Z_2 \subseteq V(G_2) \setminus C$. W.l.o.g. (possibly after renaming $G_1$ and $G_2$) we may assume from now on that $Z_1, Z_2 \subseteq V(G_1) \setminus C$. Then every $Z \in \mathcal{Z} \setminus \{Z_1,Z_2\}$ must also be linked with an edge to one of $Z_1$ or $Z_2$, and hence cannot be entirely contained in $V(G_2) \setminus C$. Hence, we have $Z \cap V(G_1) \neq \emptyset$ for every $Z \in \mathcal{Z}$. We now claim that the collection $\mathcal{Z}':=\{Z \cap V(G_1)|Z \in \mathcal{Z}\}$ of disjoint vertex-subsets in $G_1$ certifies that $G_1$ also contains a $K_{s,t}$-minor.

Firstly, for every $Z \in \mathcal{Z}$, the graph $G_1[Z \cap V(G_1)]$ is connected. Namely, we either have $Z \subseteq V(G_1)$, and hence $G_1[Z\cap V(G_1)]=G_1[Z]=G[Z]$ is a connected subgraph of $G_1$, or we have $Z \cap C \neq \emptyset$. In the latter case, the facts that $G[Z]$ is connected and that $C$ is a clique yield that $Z \cap V(G_1)$ also induces a connected subgraph of $G$ (and hence of $G_1$), as desired. 

Secondly, for every pair of distinct sets $X, Y \in \mathcal{Z}$ with $X \in \mathcal{Z}_1, Y \in \mathcal{Z}_2$, there exists an edge $e$ in $G$ with endpoints in $X$ and $Y$. If these endpoints both lie in $V(G_1)$, then the same edge links $X \cap V(G_1)$ and $Y \cap V(G_1)$ in $G_1$. Otherwise, at least one endpoint of $e$ is contained in $V(G_2)\setminus C$, and in this case the connectivity of $G[X], G[Y]$ implies that $X \cap C \neq \emptyset \neq Y \cap C$. Since $C$ is a clique, the latter directly implies that there is an edge in $G_1$ joining a vertex in $X \cap C$ to a vertex in $Y \cap C$, again certifying that $X \cap V(G_1)$ and $X \cap V(G_2)$ are linked by an edge in $G_1$. All in all, this shows that the collection $\mathcal{Z}'$ certifies the existence of a $K_{s,t}$-minor in $G_1$, a contradiction to the assumptions made in the lemma. This concludes the proof. 
\end{proof}

\begin{proof}[Proof of Theorem~\ref{main}]
Let fixed constants $\varepsilon \in (0,1)$ and $C \ge 1$ be given, and assume w.l.o.g. that $\varepsilon<\frac{1}{2}$. Define $\varepsilon':=\frac{\varepsilon}{2}$ and $C':=2C+2 \ge 1$. Let $n_0=n_0(\varepsilon', C') \in \mathbb{N}$ be chosen as in Lemma~\ref{cor} applied with parameters $\varepsilon', C'$, and define $N:=\max\{n_0+1,\lceil \frac{4}{\varepsilon} \rceil\} \in \mathbb{N}$. 

Let us now go about proving the claim of Theorem~\ref{main}. For that purpose, let $s,t$ be any given integers such that $N \le s \le t \le Cs$, and let us show that there exists a graph with no $K_{s,t}$-minor and list chromatic number greater than $(1-\varepsilon)(2s+t)$. For that purpose, define $n:=s-1$ and $m:= \lfloor (1-\varepsilon)(s+t)\rfloor$, noting that we have $n_0 \le n \le m$ as well as $m \le s+t \le (C+1)s=(C+1)(n+1) \le C'n$. 

We may therefore apply Lemma~\ref{cor} to the parameters $m$ and $n$, which yields a graph $H$ whose vertex-set is partitioned into two non-empty sets $A$ and $B$ of size $m$ and $n$ respectively, such that both $A$ and $B$ form cliques in $H$, every vertex in $H$ has at most $\varepsilon' n$ non-neighbors, and $H$ is $K_{s,t}$-minor free (since $n \le s$ and $m \le (1-\varepsilon)(s+t)=(1-2\varepsilon')(s+t)$, by definition of $m$ and $n$). 

For each possible choice of an assignment $c \in [m+n-1]^B$ of colors from $[m+n-1]$ to vertices in $B$, denote by $H(c)$ an isomorphic copy of $H$, such that the vertex-set of $H(c)$ decomposes into the cliques $A(c)$ and $B$ of size $m$ and $n$, respectively. More precisely, the distinct copies $H(c), c \in [m+n-1]^B$ of $H$ share the same set $B$ but have pairwise disjoint sets $A(c)$. Since $B$ forms a clique of size $n=s-1<s$ in the $K_{s,t}$-minor-free graph $H(c)$ for every coloring $c:B \rightarrow [m+n-1]$, it follows by repeated application of Lemma~\ref{glue} that the graph $\mathbf{G}$ with vertex set $\bigcup_{c \in [m+n-1]^B}{A(c)} \cup B$, defined as the union of the graphs $H(c), c \in [m+n-1]^B$, is $K_{s,t}$-minor free as well.  

Now, consider an assignment $L:V(\mathbf{G}) \rightarrow 2^\mathbb{N}$ of color lists to the vertices of $\mathbf{G}$ as follows:
For every vertex $b \in B$, we define $L(b):=[m+n-1]$, and for every vertex $a \in A(c)$ for some coloring $c \in [m+n-1]^B$ of $B$, we define $L(a):= [m+n-1] \setminus \{c(b)|b \in B, ab \notin E(H(c))\}$.
Note that since every vertex in $A(c)$ has at most $\varepsilon' n$ non-neighbors in $H(c)$, we have $|L(v)| \ge m+n-1-\varepsilon' n$ for every vertex $v \in V(\mathbf{G})$. 

We now claim that $\mathbf{G}$ does not admit an $L$-coloring, which will then imply the inequality $\chi_\ell(\mathbf{G}) \ge m+n-\varepsilon' n$. Indeed, suppose towards a contradiction there exists a proper coloring $c_\mathbf{G}:V(\mathbf{G}) \rightarrow \mathbb{N}$ of $\mathbf{G}$ such that $c_\mathbf{G}(v) \in L(v)$ for every $v \in V(\mathbf{G})$. Let $c$ denote the restriction of $c_\mathbf{G}$ to $B$, and consider the proper coloring of $H(c)$ obtained by restricting $c_\mathbf{G}$ to the vertices in $H(c)$. Since $v(H(c))=m+n$ and $c_\mathbf{G}(v) \in [m+n-1]$ for every $v \in V(H(c))$, there must exist two (necessarily non-adjacent) vertices in $H(c)$ which have the same color with respect to $c_\mathbf{G}$. Concretely, there exist $a \in A(c)$, $b\in B$ such that $ab \notin E(H(c))$ and $c_\mathbf{G}(a)=c_\mathbf{G}(b)$. This however yields a contradiction, since $c_\mathbf{G}(a) \in L(a)$ and by definition $c(b)=c_\mathbf{G}(b)$ is not included in the list of $a$. 

We conclude that indeed, $\mathbf{G}$ is a $K_{s,t}$-minor-free graph which satisfies $$\chi_\ell(\mathbf{G}) \ge m+n-\varepsilon' n=\lfloor (1-\varepsilon)(s+t) \rfloor+\left(1-\frac{\varepsilon}{2}\right)(s-1)$$ $$>(1-\varepsilon)(s+t)-1+(1-\varepsilon)s+\frac{\varepsilon}{2}s-\left(1-\frac{\varepsilon}{2}\right)$$
$$=(1-\varepsilon)(2s+t)+\frac{\varepsilon}{2}s-2+\frac{\varepsilon}{2}>(1-\varepsilon)(2s+t),$$
where for the last inequality we used that $s \ge N \ge \frac{4}{\varepsilon}$. 
\end{proof}

\end{document}